\documentclass[12pt,reqno]{amsart}

\usepackage{amssymb}

\newtheorem{theorem}{Theorem}[section]
\newtheorem{lemma}[theorem]{Lemma}
\newtheorem{proposition}[theorem]{Proposition}
\newtheorem{corollary}[theorem]{Corollary}

\theoremstyle{remark}
\newtheorem{remark}[theorem]{Remark}

\numberwithin{equation}{section}

\begin{document}

\title[Moduli space of holomorphic $G$-connections]{On the moduli space of holomorphic 
$G$-connections on a compact Riemann surface}

\author[I. Biswas]{Indranil Biswas}

\address{School of Mathematics, Tata Institute of Fundamental Research,
Homi Bhabha Road, Mumbai 400005, India}

\email{indranil@math.tifr.res.in}



\date{}

\begin{abstract}
Let $X$ be a compact connected Riemann surface of genus at least two and $G$ a connected
reductive complex affine algebraic group. The Riemann--Hilbert correspondence produces
a biholomorphism between the moduli space ${\mathcal M}_X(G)$ parametrizing holomorphic
$G$--connections on $X$ and the $G$--character variety
$${\mathcal R}(G)\,:=\, \text{Hom}(\pi_1(X, x_0),\, G)/\!\!/G\, .$$ While ${\mathcal R}(G)$
is known to be affine, we show that ${\mathcal M}_X(G)$ is not affine. The scheme ${\mathcal R}(G)$
has an algebraic symplectic form constructed by Goldman. We construct an algebraic
symplectic form on ${\mathcal M}_X(G)$ with the property that the Riemann--Hilbert correspondence
pulls back to the Goldman symplectic form to it. Therefore, despite the Riemann--Hilbert
correspondence being non-algebraic, the pullback of the Goldman symplectic form by
the Riemann--Hilbert correspondence nevertheless continues to be algebraic.
\end{abstract}

\maketitle

\section{Introduction}\label{sec0}

Let $X$ be a compact connected Riemann surface of genus at least two. Take a connected 
reductive complex affine algebraic group $G$. Let ${\mathcal M}_X(G)$ denote the moduli space
of pairs $(E_G,\, D)$, where $E_G$ is a holomorphic principal $G$--bundle on $X$ and
$D$ is a holomorphic connection on $E_G$. This moduli space ${\mathcal M}_X(G)$ is a
normal scheme, but it need not be connected. The tangent space to ${\mathcal M}_X(G)$ at
a point $(E_G,\, D)$ is given by the first hypercohomology of a two-term complex given by the
holomorphic connection on the adjoint vector bundle $\text{ad}(E_G)$ induced by $D$. Using this
description of the tangent space, we construct an algebraic symplectic form on
${\mathcal M}_X(G)$, which is denoted by $\Theta$ (see Lemma \ref{lem1} and Corollary
\ref{cor1}).

Fix a point $x_0\, \in\, X$. Let ${\mathcal R}(G)\,:=\, \text{Hom}(\pi_1(X, x_0),\, G)/\!\!/G$
be the $G$--character variety associated to the pair $(X,\, G)$. This ${\mathcal R}(G)$
is a normal affine scheme though it may not be connected. In fact, the connected
components of ${\mathcal R}(G)$ (as well as those of ${\mathcal M}_X(G)$) are parametrized
by the torsion part of $\pi_1(G)$. The scheme ${\mathcal R}(G)$, by construction, is affine.
The Riemann--Hilbert correspondence produces a biholomorphism $\Phi\, :\, {\mathcal M}_X(G)\, 
\longrightarrow\, {\mathcal R}(G)$; it sends any $(E_G,\, D)$ to the monodromy representation 
for the flat connection on $E_G$ constructed using $D$ and the Dolbeault operator defining
the holomorphic 
structure of $E_G$. We prove that the scheme ${\mathcal M}_X(G)$ is not affine
(Proposition \ref{prop1}). In particular, the Riemann--Hilbert correspondence is not algebraic.

Goldman, \cite{Go}, constructed an algebraic symplectic form $\Theta_{\mathcal R}$ on 
${\mathcal R}(G)$. As noted above, the map Riemann--Hilbert correspondence $\Phi$ is not algebraic.
However, the surprising 
fact is that the pullback $\Phi^*\Theta_{\mathcal R}$ of the algebraic symplectic form 
$\Theta_{\mathcal R}$ remains algebraic. To be precise, the form $\Phi^*\Theta_{\mathcal R}$ 
coincides with the algebraic form $\Theta$ on ${\mathcal M}_X(G)$ (see Theorem \ref{thm1}).

\section{Moduli space of connections and character variety}

Let $G$ be a connected reductive affine algebraic group defined over $\mathbb C$. The Lie 
algebra of $G$ will be denoted by $\mathfrak g$. A parabolic subgroup of $G$ is a
Zariski closed connected subgroup $P\, \subset\, G$ such that the quotient $G/P$
is a projective variety. The unipotent radical of a parabolic subgroup $P$ will be denoted by
$R_u(P)$. The quotient group $P/R_u(P)$ is called the Levi quotient of $P$ (see
\cite[p. 158, \S~11.23]{Bo}, \cite[\S~30.2, p. 184]{Hu}). The
center of $G$ will be denoted by $Z_G$.

Let $X$ be a compact connected Riemann surface of genus $g$. The
holomorphic cotangent bundle of $X$ will be denoted by $K_X$.

The holomorphic tangent bundle of any complex manifold $Z$ will be denoted
by $TZ$.

Take a holomorphic principal $G$--bundle
\begin{equation}\label{e1}
p\, :\, E_G\,\longrightarrow\, X
\end{equation}
over $X$. Consider
the action of $G$ on $TE_G$ given by the action of $G$ on $E_G$. The quotient
$$
\text{At}(E_G)\,:=\, (TE_G)/G
$$
is a holomorphic vector bundle over $E_G/G\,=\, X$; it is called the
\textit{Atiyah bundle} for $E_G$. Consider the differential $$dp\, :\,
TE_G\, \longrightarrow\, p^* TX$$ of the projection $p$ in \eqref{e1}. Note that
$G$ has a natural action on $p^*TX$ because it is
pulled back from $E_G/G\,=\, X$. The above homomorphism $dp$ is
evidently equivariant for the actions of $G$ on $TE_G$ and $p^*TX$. Therefore,
it descends to a homomorphism
\begin{equation}\label{e2}
p'\, :\, \text{At}(E_G)\,:=\, (TE_G)/G\, \longrightarrow\, (p^* TX)/G \,=\, TX\, .
\end{equation}
This homomorphism $p'$ is surjective because $dp$ is so. So we have
$$
\text{kernel}(p')\,=\, \text{kernel}(dp)/G \,=\, T_{E_G/X}/G\, ,
$$
where $T_{E_G/X}\,=\, \text{kernel}(dp)\, \longrightarrow\, E_G$ is the relative tangent bundle for the
projection $p$.

Using the action of $G$ on $E_G$, the relative tangent bundle $T_{E_G/X}$ is identified
with the trivial vector bundle $E_G\times {\mathfrak g}\,\longrightarrow\, E_G$
with fiber $\mathfrak g\,:=\, \text{Lie}(G)$. Consequently,
the quotient $T_{E_G/X}/G$ gets identified with the vector bundle over $X$ associated to
the principal $G$--bundle $E_G$ for the adjoint action of $G$ on $\mathfrak g$. This
associated vector bundle, which is denoted by $\text{ad}(E_G)$, is called the adjoint bundle
for $E_G$.

Therefore, we have a short exact sequence of holomorphic vector bundles on $X$
\begin{equation}\label{e3}
0\, \longrightarrow\, \text{ad}(E_G)\, \longrightarrow\, \text{At}(E_G)\,
\stackrel{p'}{\longrightarrow}\, TX \,\longrightarrow\, 0
\end{equation}
(see \cite[p.~187, Theorem 1]{At}). The sequence in \eqref{e3} is known as the \textit{Atiyah exact
sequence} for $E_G$. A \textit{holomorphic connection} on $E_G$ is a holomorphic homomorphism
of vector bundles
$$
D\, :\, TX \, \longrightarrow\, \text{At}(E_G)
$$
such that $p'\circ D\,=\, \text{Id}_{TX}$ \cite[p.~188, Definition]{At}.

A holomorphic $G$--\textit{connection} on $X$ is a pair $(E_G,\, D)$, where $E_G$ is a 
holomorphic principal $G$--bundle on $X$, and $D$ is a holomorphic connection on $E_G$.

The curvature of a holomorphic connection on $E_G$ vanishes identically because
$\Omega^{2,0}_X\,=\, 0$. Therefore, holomorphic $G$--connections correspond to
homomorphisms from the fundamental group of $X$ to $G$
(see \cite[p.~200, Proposition 14]{At}). In particular, $E_G$ admits a holomorphic connection if
and only if $E_G$ is given by some homomorphism from the fundamental group of $X$ to $G$.

A holomorphic $G$--connection $(E_G,\, D)$ on $X$ is called \textit{reducible} if there is a 
proper parabolic subgroup $P\, \subsetneq\, G$, and a holomorphic $P$--connection $(E_P,\, D_P)$ on 
$X$, such that $(E_G,\, D)$ is the extension of structure group of $(E_P,\, D_P)$ using the
inclusion of $P$ in $G$. Note that there is a pair $(E_P,\, D_P)$ satisfying this condition
if and only if $(E_G,\, D)$ is given by a homomorphism from the fundamental group of $X$ to $P$.
A holomorphic $G$--connection $(E_G,\, D)$ on $X$ is called \textit{irreducible} if it
is not reducible.

Let ${\mathcal M}_X(G)$ denote the moduli space of holomorphic $G$--connections on $X$ \cite{Ni},
\cite{Si1}, \cite{Si2}. Two holomorphic $G$--connections $(E_G,\, D)$ and $(E'_G,\, D')$ are called
$S$--equivalent if there is a parabolic subgroup $P\, \subset\, G$ such that
\begin{enumerate}
\item there are holomorphic $P$--connections $(E_P,\, D_P)$ and $(E'_P,\, D'_P)$ such that
$(E_G,\, D)$ (respectively, $(E'_G,\, D')$) is the extension of structure group of
$(E_P,\, D_P)$ (respectively, $(E'_P,\, D'_P)$) using the
inclusion of $P$ in $G$, and

\item the holomorphic $P/R_u(P)$--connections $(E_P/R_u(P),\, \widetilde{D}_P)$ and
$(E'_P/R_u(P),\, \widetilde{D}'_P)$ are isomorphic, where $\widetilde{D}_P$
(respectively, $\widetilde{D}'_P$) is the connection on the principal
$L(P)$--bundle $E_P/R_u(P)$ (respectively, $E'_P/R_u(P)$) induced by
$D_P$ (respectively, $D'_P$). (Note that if $E_H$ is a principal $H$--bundle and $N$ is a normal subgroup
of $H$, then the quotient $E_H/N$ is a principal $H/N$--bundle.)
\end{enumerate}
The points of the moduli space ${\mathcal M}_X(G)$ parametrize all the $S$--equivalence classes of
holomorphic $G$--connections.

The moduli space ${\mathcal M}_X(G)$ is a reduced normal quasiprojective scheme defined over
$\mathbb C$. Its connected
components are parametrized by the torsion elements of the fundamental group $\pi_1(G)$.
Each connected component of ${\mathcal M}_X(G)$ is irreducible of dimension
$2\cdot \dim G/Z_G\cdot (g-1)+2g\cdot \dim Z_G\,=\, 2\cdot \dim G\cdot (g-1)+2\cdot \dim Z_G$ if $g\, \geq\, 2$,
where $Z_G$ is the center of $G$. The moduli space ${\mathcal M}_X(G)$ is
a singleton if $g\,=\, 0$; indeed, it is the trivial principal $G$--bundle equipped
with the trivial connection, because ${\mathbb C}{\mathbb P}^1$ is simply connected.

Let
$$
{\mathcal M}^0_X(G)\, \subset\, {\mathcal M}_X(G)
$$
be the moduli space of irreducible holomorphic $G$--connections on $X$. This ${\mathcal M}^0_X(G)$
is a Zariski open subset of ${\mathcal M}_X(G)$; it is also dense
if $g\, \geq\, 2$. Note that if two
irreducible holomorphic $G$--connections are $S$--equivalent, then they are actually isomorphic.
All the singular points of ${\mathcal M}^0_X(G)$ have finite quotient (orbifold)
singularity.

Henceforth, we shall always assume that $g\, \geq\, 2$.

The $C^\infty$ oriented (real) surface underlying $X$ will be denoted by $X_0$; this notation
will be used when the complex structure of $X$ is not relevant. Fix a point $x_0\, \in\, X_0$.
Let
$$
{\mathcal R}(G)\,:=\, \text{Hom}(\pi_1(X_0, x_0),\, G)/\!\!/G
$$
be the $G$--character variety associated to $X_0$. Since the group $\pi_1(X_0, x_0)$ is
finitely presented, using the complex algebraic structure of $G$, the moduli space
${\mathcal R}(G)$ becomes an affine scheme defined over $\mathbb C$. In fact, each connected
component of ${\mathcal R}(G)$ is an irreducible
normal affine variety defined over $\mathbb C$. The connected
components of ${\mathcal R}(G)$ are parametrized by the torsion elements of $\pi_1(G)$.
The moduli space ${\mathcal R}(G)$ also coincides with the moduli space of flat
principal $G$--bundles on $X_0$.

The Riemann--Hilbert correspondence produces a biholomorphism
\begin{equation}\label{e4}
\Phi\,:\, {\mathcal M}_X(G)\, \longrightarrow\, {\mathcal R}(G)\, .
\end{equation}
This map $\Phi$ sends a holomorphic $G$--connection $(E_G,\, D)$ to the monodromy
of the flat connection on $E_G$ defined by $D$ and the holomorphic structure of $E_G$.

\begin{proposition}\label{prop1}
The complex scheme ${\mathcal M}_X(G)$ is not affine. In particular,
the two complex schemes ${\mathcal M}_X(G)$ and ${\mathcal R}(G)$ are not
algebraically isomorphic.
\end{proposition}

\begin{proof}
Since ${\mathcal R}(G)$ is an affine scheme, the first statement in the proposition implies
the second statement.

The multiplicative group of nonzero complex numbers will be denoted by ${\mathbb C}^*$.
Let ${\mathcal M}_X({\mathbb C}^*)$ be the moduli space of holomorphic ${\mathbb C}^*$--connections
on $X$. Fix an algebraic embedding
\begin{equation}\label{e6}
\rho\, :\, {\mathbb C}^*\, \longrightarrow\, G\, .
\end{equation}
It produces an algebraic embedding
\begin{equation}\label{e5}
\widehat{\rho}\, :\, {\mathcal M}_X({\mathbb C}^*)\, \longrightarrow\,
{\mathcal M}_X(G)
\end{equation}
which sends a holomorphic ${\mathbb C}^*$--connection $(L, \, D_L)$ on $X$
to the holomorphic $G$--connection $(L(\rho), \, D_L(\rho))$, where $L(\rho)$ is the
holomorphic principal $G$--bundle obtained by extending the structure group of $L$
using the homomorphism $\rho$ in \eqref{e6}, and $D_L(\rho)$ is the holomorphic connection
on $L(\rho)$ induced by the connection $D_L$.

Let $\text{Pic}^0(X)$ denote the Jacobian of $X$ that parametrizes all isomorphism
classes of holomorphic line bundles on $X$ of degree zero. Fix a point $x_0\, \in\, X$. Let
$$
\text{AJ}_X\, :\, X\, \longrightarrow\, \text{Pic}^0(X)\, ,\ \ x\, \longmapsto\,
{\mathcal O}_X(x-x_0)
$$
be the Abel--Jacobi embedding. Let ${\mathcal M}(\text{Pic}^0(X))$ be the moduli space of
integrable holomorphic connections on $\text{Pic}^0(X)$ of rank one. We have an algebraic morphism
$$
\psi\, :\, {\mathcal M}(\text{Pic}^0(X))\, \longrightarrow\,
{\mathcal M}_X({\mathbb C}^*)\, ,\ \ (L,\, D_L)\, \longmapsto\, (\text{AJ}^*_XL,\,
\text{AJ}^*_X D_L)\, ;
$$
here we have identified the principal ${\mathbb C}^*$--bundles with line bundles
using the multiplicative action of ${\mathbb C}^*$ on ${\mathbb C}$. It is known that
$\psi$ is an isomorphism. A quick way to see this is as follows: the homomorphism of fundamental groups
$$\text{AJ}_{X,*}\, :\,
\pi_1(X,\, x_0)\, \longrightarrow\, \pi_1(\text{Pic}^0(X),\, {\mathcal O}_X)$$
induced by $\text{AJ}_X$ is the abelianization of $\pi_1(X,\, x_0)$, and hence it identifies
the characters of $\pi_1(\text{Pic}^0(X),\, {\mathcal O}_X)$ 
with the characters of $\pi_1(X,\, x_0)$.

It is known that there are no non-constant algebraic functions on ${\mathcal M}(\text{Pic}^0(X))$
\cite[p.~1541, Proposition 4.1]{BHR}. Since $\psi$ is an algebraic isomorphism, it follows
that ${\mathcal M}_X({\mathbb C}^*)$ does not admit any non-constant algebraic function.
Therefore, in view of the embedding $\widehat{\rho}$ in \eqref{e5} we conclude that
the scheme ${\mathcal M}_X(G)$ is not affine.
\end{proof}

\begin{remark}\label{rem0}
Proposition \ref{prop1} is known for the special case of $G\,=\, \text{SL}(r, {\mathbb C})$
for all $r\, \geq\, 2$. In fact, there are no non-constant algebraic functions
on ${\mathcal M}_X(\text{SL}(r, {\mathbb C}))$ \cite[p.~803, Theorem 4.5]{BR}.
\end{remark}

\section{A symplectic form on the moduli spaces}

\subsection{A two-form on the moduli space of holomorphic $G$-connections}

Take a point $$(E_G,\, D)\, \in\, {\mathcal M}_X(G)\, .$$ As before, let $\text{ad}(E_G)\,=\,
E_G\times^G{\mathfrak g}\,\longrightarrow\, X$ be the adjoint vector bundle for the principal
$G$--bundle $E_G$. The connection on $\text{ad}(E_G)$ induced by the connection $D$ on $E_G$ 
will be denoted by ${\mathcal D}$. We note that $\mathcal D$ is a holomorphic differential
operator $\text{ad}(E_G)\,\longrightarrow\, \text{ad}(E_G)\otimes K_X$ of order one satisfying
the Leibniz identity which says that $${\mathcal D}(fs)\,=\, f\cdot {\mathcal D}(s) +
s\otimes df$$ for any locally defined holomorphic section $s$ of $\text{ad}(E_G)$ and any
locally defined holomorphic function $f$ on $X$.

We have a two term complex of coherent sheaves on $X$
\begin{equation}\label{e7}
{\mathcal C}_{\bullet}\, :\,
{\mathcal C}_{0}\,=\, \text{ad}(E_G) \, \stackrel{\mathcal D}{\longrightarrow}\,
{\mathcal C}_{1}\,=\, \text{ad}(E_G)\otimes K_X\, ,
\end{equation}
where ${\mathcal C}_i$ is at the $i$-th position.
The infinitesimal deformations of the holomorphic $G$--connection $(E_G,\, D)$ are parametrized
by the hypercohomology ${\mathbb H}^1({\mathcal C}_{\bullet})$ \cite[p.~606, Theorem 4.2]{Ni}.

Since $G$ is reductive, its Lie algebra $\mathfrak g$ admits a $G$--invariant nondegenerate symmetric
bilinear form. Fix such a form
\begin{equation}\label{B}
B\, \in\, \text{Sym}^2({\mathfrak g}^*)\, .
\end{equation}
Since $B$ is
$G$--invariant, it produces a fiberwise nondegenerate symmetric bilinear form
\begin{equation}\label{Bn}
\widetilde{B}\, \in\, H^0(X,\, \text{Sym}^2(\text{ad}(E_G)^*))
\end{equation}
on the vector bundle $\text{ad}(E_G)$. This $\widetilde{B}$ produces an isomorphism
of $\text{ad}(E_G)$ with its dual $\text{ad}(E_G)^*$. Whenever $\text{ad}(E_G)$ will be
identified with $\text{ad}(E_G)^*$, it should understood that $\widetilde{B}$ is being
used. The Serre dual complex for ${\mathcal C}_{\bullet}$ is ${\mathcal C}_{\bullet}$ itself
because $$(\text{ad}(E_G)\otimes K_X)^*\otimes K_X\,=\, \text{ad}(E_G)\ \ \text{ and }\ \
\text{ad}(E_G)^*\otimes K_X\,=\, \text{ad}(E_G)\otimes K_X\, .$$ Hence Serre duality in this case
produces an isomorphism
\begin{equation}\label{e8}
\beta\, :\, {\mathbb H}^1({\mathcal C}_{\bullet})\, \longrightarrow\,
{\mathbb H}^1({\mathcal C}_{\bullet})^*\, .
\end{equation}
To describe $\beta$ in \eqref{e8} more explicitly, consider the
tensor product of complexes $\mathcal{C}_{\bullet}\otimes \mathcal{C}_{\bullet}$:
$$
(\mathcal{C}_{\bullet}\otimes \mathcal{C}_{\bullet})_0\,=\,
\text{ad}(E_G)\otimes \text{ad}(E_G)\, \stackrel{({\mathcal D}\otimes {\rm Id})+
({\rm Id}\otimes{\mathcal D})}{\longrightarrow}\,
(\mathcal{C}_{\bullet}\otimes \mathcal{C}_{\bullet})_1
$$
$$
=\, ((\text{ad}(E_G)\otimes K_X)\otimes \text{ad}(E_G))\oplus
(\text{ad}(E_G)\otimes (\text{ad}(E_G)\otimes K_X))
$$
$$
\stackrel{({\rm Id}\otimes {\mathcal D})-
({\mathcal D}\otimes {\rm Id})}{\longrightarrow}\,
(\mathcal{C}_{\bullet}\otimes \mathcal{C}_{\bullet})_2\,=\,
(\text{ad}(E_G)\otimes K_X)\otimes (\text{ad}(E_G)\otimes K_X)\, ,
$$
where $(\mathcal{C}_{\bullet}\otimes \mathcal{C}_{\bullet})_i$ is at the $i$-th position.
We also have the homomorphisms
$$
\gamma_0\, :\, \text{ad}(E_G)\otimes \text{ad}(E_G)\, \longrightarrow\, {\mathcal O}_X\, ,
\ \ a\otimes b \, \longmapsto\, {\widetilde B}(a\otimes b)\, ,
$$
where ${\widetilde B}$ is constructed in \eqref{Bn}, and
$$
\gamma_1\, :\, ((\text{ad}(E_G)\otimes K_X)\otimes \text{ad}(E_G))\oplus
(\text{ad}(E_G)\otimes (\text{ad}(E_G)\otimes K_X))
$$
$$
\longrightarrow\, K_X\, , \ \ (a_1\otimes b_1)\oplus (a_2\otimes b_2)\, \longmapsto\,
{\widetilde B}(a_1\otimes b_1)+{\widetilde B}(a_2\otimes b_2)\, .
$$

Let ${\mathbb K}$ denote the complex of coherent sheaves on $X$
$${\mathcal O}_X\,\stackrel{d}{\longrightarrow}\, K_X
\,\longrightarrow\, 0\, ,$$
where ${\mathcal O}_X$ and $K_X$ are at the $0$-th position and $1$-position respectively, and
$d$ is the de Rham differential.

It is straight-forward to check that $\gamma_1\circ (({\mathcal D}\otimes {\rm Id})+
({\rm Id}\otimes{\mathcal D}))\,=\, d\circ\gamma_0$. Consequently,
the above homomorphisms $\gamma_0$ and $\gamma_1$ produce a homomorphism $\gamma$ of complexes
$$
\begin{matrix}
\mathcal{C}_{\bullet}\otimes \mathcal{C}_{\bullet} & : &
(\mathcal{C}_{\bullet}\otimes \mathcal{C}_{\bullet})_0 & {\longrightarrow} &
(\mathcal{C}_{\bullet}\otimes \mathcal{C}_{\bullet})_1 & \longrightarrow &
(\mathcal{C}_{\bullet}\otimes \mathcal{C}_{\bullet})_2\\
~\Big\downarrow\gamma && ~\,\,\,\Big\downarrow\gamma_0 && ~\,\,\,\Big\downarrow\gamma_1 && \Big\downarrow\\
{\mathbb K} & : & {\mathcal O}_X & \stackrel{d}{\longrightarrow} & K_X & \longrightarrow & 0
\end{matrix}
$$
{}From this we have the composition of homomorphisms of hypercohomologies
\begin{equation}\label{e9}
{\mathbb H}^1({\mathcal C}_{\bullet})\otimes
{\mathbb H}^1({\mathcal C}_{\bullet})\, \longrightarrow\, {\mathbb H}^2(\mathcal{C}_{\bullet}\otimes \mathcal{C}_{\bullet})
\, \stackrel{\gamma_*}{\longrightarrow}\,{\mathbb H}^2({\mathbb K})\, ,
\end{equation}
where $\gamma_*$ is induced by the above homomorphism $\gamma$ of complexes.

We will show that
\begin{equation}\label{e9a}
{\mathbb H}^2({\mathbb K})\,=\, H^1(X,\, K_X)\,=\, {\mathbb C}\, .
\end{equation}
For this first consider the following short exact sequence of complexes of sheaves on $X$
$$
\begin{matrix}
&& 0 & & 0\\
&& \Big\downarrow & & \Big\downarrow \\
&& 0 & \longrightarrow & K_X\\
&& \Big\downarrow && \,\,\,\,\,\,\,\Big\downarrow=\\
{\mathbb K} & : & {\mathcal O}_X & \stackrel{d}{\longrightarrow} & K_X\\
&& \,\,\,\,\,\,\,\Big\downarrow= && \Big\downarrow\\
&& {\mathcal O}_X & \longrightarrow & 0\\
& & \Big\downarrow & & \Big\downarrow \\
& & 0 & & 0
\end{matrix}
$$
This produces a long exact sequence of hypercohomologies
\begin{equation}\label{e9b}
\longrightarrow\, H^1(X,\, {\mathcal O}_X)\, \stackrel{d_*}{\longrightarrow}\, H^1(X,\, K_X)
\, \longrightarrow\, {\mathbb H}^2({\mathbb K}) \, \longrightarrow\, H^2(X,\, {\mathcal O}_X)\,=\, 0\, ,
\end{equation}
where the homomorphism $d_*$ is the homomorphism of cohomologies induced by $d$. By Serre duality,
we have
\begin{equation}\label{int0}
H^1(X,\, K_X)\,=\, H^0(X,\, {\mathcal O}_X)^*\,=\, \mathbb C\, .
\end{equation}
To describe the isomorphism $H^1(X,\, K_X)\,\stackrel{\sim}{\longrightarrow}\, \mathbb C$
in \eqref{int0} explicitly, first note that using Dolbeault approach,
$$
H^1(X,\, K_X)\,=\, \frac{C^\infty(X,\, \Omega^{1,1}_X)}{\overline{\partial}(C^\infty(X,\, \Omega^{1,0}_X))}\, .
$$
Now consider the homomorphism
\begin{equation}\label{int}
C^\infty(X,\, \Omega^{1,1}_X)\, \longrightarrow\, {\mathbb C}\, ,\ \ \omega\, \longmapsto\, \int_X\omega\, .
\end{equation}
{}From Stokes' theorem we know that the homomorphism in \eqref{int} vanishes on
$$\overline{\partial}(C^\infty(X,\, \Omega^{1,0}_X))\,=\, d(C^\infty(X,\, \Omega^{1,0}_X))$$
(the above equality is a consequence of the fact that ${\partial}(C^\infty(X,\, \Omega^{1,0}_X))\,=\, 0$
as the sheaf of $(2,\, 0)$-forms on $X$ is the zero sheaf). Therefore,
the homomorphism in \eqref{int} produces a homomorphism
\begin{equation}\label{int2}
H^1(X,\, K_X)\,=\, \frac{C^\infty(X,\, \Omega^{1,1}_X)}{\overline{\partial}(C^\infty(X,\, \Omega^{1,0}_X))}
\, \longrightarrow\, {\mathbb C}\, .
\end{equation}
This homomorphism in \eqref{int2} is an isomorphism, and it coincides with the one in \eqref{int0}.

Using Dolbeault approach,
$$
H^1(X,\, {\mathcal O}_X)\,=\,
\frac{C^\infty(X,\, \Omega^{0,1}_X)}{\overline{\partial}(C^\infty(X,\, {\mathbb C}))}\, .
$$
For any $\alpha \,\in\, C^\infty(X,\, \Omega^{0,1}_X)$, from Stokes' theorem we know that

$$
\int_X d(\alpha)\, =\, 0\, .
$$
Consequently, the homomorphism $d_*$ in \eqref{e9b} is the zero homomorphism; here we are
using the description of $H^1(X,\, K_X)$ given by \eqref{int2}. Hence
\eqref{e9a} follows from \eqref{e9b}.

Combining \eqref{e9} and \eqref{e9a} we get a homomorphism from ${\mathbb H}^1({\mathcal C}_{\bullet})\otimes
{\mathbb H}^1({\mathcal C}_{\bullet})$ to $\mathbb C$. Let
\begin{equation}\label{e10}
\Theta(E_G,D)\, :\, {\mathbb H}^1({\mathcal C}_{\bullet})\otimes
{\mathbb H}^1({\mathcal C}_{\bullet})\, \longrightarrow\, \mathbb C
\end{equation}
denote this bilinear form. The bilinear form $\Theta(E_G,D)$
in \eqref{e10} produces $\beta$ in \eqref{e8} using the
equation
\begin{equation}\label{f1}
\beta (\omega_1)(\omega_2) \,=\, \Theta(E_G,D)(\omega_1\otimes \omega_2)
\end{equation}
for all $\omega_1,\, \omega_2\, \in\, {\mathbb H}^1({\mathcal C}_{\bullet})$.

{}From the above
construction of $\Theta(E_G,D)$ it is evident that the bilinear form $\Theta(E_G,D)$ is 
anti-symmetric. Since $\beta$ is an isomorphism, from \eqref{f1} it follows immediately that 
the anti-symmetric pairing $\Theta(E_G,D)$ is nondegenerate.

As before,
$$
{\mathcal M}^0_X(G)\, \subset\, {\mathcal M}_X(G)
$$
is the moduli space of irreducible holomorphic
$G$--connections. We noted above that
$$T_{(E_G, D)} {\mathcal M}^0_X(G)\,=\, {\mathbb H}^1({\mathcal C}_{\bullet})$$
for $(E_G,\, D)\, \in\, {\mathcal M}^0_X(G)$. The construction of $\Theta(E_G,D)$ in \eqref{e10}
evidently works for families of holomorphic $G$--connections.
Therefore, the point-wise construction of $\Theta(E_G,D)$ produces an algebraic two-form
on ${\mathcal M}^0_X(G)$. So we have the following:

\begin{lemma}\label{lem1}
The moduli space ${\mathcal M}^0_X(G)$ has a natural algebraic two-form
$$
\Theta\, \in\, H^0({\mathcal M}^s_X(G),\, \Omega^2_{{\mathcal M}^s_X(G)})
$$
which is $\Theta(E_G,D)$ in \eqref{e10} at every $(E_G,\, D)\, \in\, {\mathcal M}^0_X(G)$.
For every $y\, \in\, {\mathcal M}^0_X(G)$, this two-form $\Theta(y)$ on
on $T_y {\mathcal M}^0_X(G)$ is nondegenerate.
\end{lemma}

\subsection{Goldman symplectic form on the character variety}\label{se3.2}

A homomorphism $\alpha \,\in\, \text{Hom}(\pi_1(X_0, x_0),\, G)$ is called \textit{irreducible}
if the image of $\alpha$ is not contained in some proper parabolic subgroup of $G$. Let
$$
{\mathcal R}^0(G)\, \subset\, {\mathcal R}(G)
$$
be the moduli space of irreducible representations; it is a Zariski open dense subset of
${\mathcal R}(G)$ (recall that $g\, \geq\,2$ by assumption). The biholomorphism $\Phi$ in \eqref{e4}
takes ${\mathcal M}^0_X(G)$ surjectively to ${\mathcal R}^0(G)$.

In \cite{Go}, Goldman constructed a complex symplectic form
\begin{equation}\label{e11}
\Theta_{\mathcal R}\, \in\, H^0({\mathcal R}^0(G),\, \Omega^2_{{\mathcal R}^0(G)})
\end{equation}
on ${\mathcal R}^0(G)$ which is in fact algebraic \cite[p.~208, Theorem]{Go} (see also
\cite{GHJW}). This symplectic form coincides with the symplectic form on the moduli space
of irreducible flat $G$--bundles on $X_0$ constructed in \cite{AB}. We shall briefly recall the description
of $\Theta_{\mathcal R}$ from the point of view of \cite{AB} (as before, $X_0$
denotes the $C^\infty$ oriented (real) surface underlying $X$).

The moduli space ${\mathcal R}^0(G)$ will always be identified with the
moduli space of irreducible flat $G$--bundles on $X_0$.

Take any irreducible flat $G$--bundle $(F_G,\,\nabla)\, \in\, {\mathcal R}^0(G)$. The flat 
connection on the adjoint bundle $\text{ad}(F_G)$ induced by the flat connection $\nabla$ on 
$F_G$ will be denoted by $\nabla^{\rm ad}$. So $\nabla^{\rm ad}$ is a $C^\infty$ differential 
operator of order one
$$
\nabla^{\rm ad}\, :\, \text{ad}(F_G)\, \longrightarrow\,\text{ad}(F_G)\otimes
T^*_{\mathbb R}X_0
$$
satisfying the Leibniz identity such that the curvature, namely the composition
$$\nabla^{\rm ad}\circ \nabla^{\rm ad}\, :\, \text{ad}(F_G)\, \longrightarrow\,\text{ad}(F_G)\otimes
\bigwedge\nolimits^2 T^*_{\mathbb R}X_0\, ,
$$
vanishes identically. Let
$$
\underline{\text{ad}}(F_G)\, \longrightarrow\, X_0
$$
be the $\mathbb C$--local system on $X_0$ given by the sheaf of flat sections of $\text{ad}(F_G)$ for
the flat connection $\nabla^{\rm ad}$.

We have
\begin{equation}\label{e12}
T_{(F_G,\nabla)} {\mathcal R}^0(G)\,=\, H^1(X_0,\, \underline{\text{ad}}(F_G))\, .
\end{equation}
As before, let $$\widetilde{B}\, \in\, C^\infty(X_0,\,
\text{Sym}^2(\text{ad}(F_G)^*))$$ be the fiber-wise nondegenerate symmetric bilinear form
on $\text{ad}(F_G)$ given by the form $B$ in \eqref{B}. Now consider the composition
\begin{equation}\label{e13}
H^1(X_0,\, \underline{\text{ad}}(F_G))\otimes H^1(X_0,\, \underline{\text{ad}}(F_G))\,
\longrightarrow\, H^2(X_0,\, \underline{\text{ad}}(F_G)\otimes \underline{\text{ad}}(F_G))
\, \stackrel{\widetilde B}{\longrightarrow}\, H^2(X_0,\,{\mathbb C})\,=\, {\mathbb C}\, ;
\end{equation}
note that the orientation of $X_0$ is used in identifying $H^2(X_0,\,{\mathbb C})$ with
$\mathbb C$. Using the identification in \eqref{e12}, the
pairing in \eqref{e13} coincides with $\Theta_{\mathcal R}(F_G,\nabla)\, \in\,
\Omega^2_{{\mathcal R}^0(G)}\vert_{(F_G,\nabla)}$ in \eqref{e11}. 

The pairing in \eqref{e13} can be explicitly described as follows. Consider the complex of
vector spaces
\begin{equation}\label{d1}
C^{\infty}(X_0,\, \text{ad}(F_G))\, \stackrel{\nabla^{\rm ad}}{\longrightarrow}\,
C^{\infty}(X_0,\, \text{ad}(F_G)\otimes T^*_{\mathbb R}X_0)
\, \stackrel{\nabla^{\rm ad}}{\longrightarrow}\,
C^{\infty}(X_0,\, \text{ad}(F_G)\otimes \bigwedge\nolimits^2 T^*_{\mathbb R}X_0)\, .
\end{equation}
For this complex, we have
\begin{equation}\label{e14}
\frac{\text{kernel}(\nabla^{\rm ad})}{\text{image} (\nabla^{\rm ad})}\,=\,
H^1(X_0,\, \underline{\text{ad}}(F_G))\, .
\end{equation}
The anti-symmetric bilinear form on $\text{kernel}(\nabla^{\rm ad})\, \subset\,
C^{\infty}(X_0,\, \text{ad}(F_G)\otimes T^*_{\mathbb R}X_0)$ defined by
\begin{equation}\label{o12}
\omega_1\otimes \omega_2\, \longrightarrow\, \int_{X_0} \widetilde{B}(\omega_1\wedge
\omega_2)\, \in\, \mathbb C
\end{equation}
descends to a bilinear form on the quotient space $\text{kernel}(\nabla^{\rm ad})/\text{image} 
(\nabla^{\rm ad})$ in \eqref{e14}. The bilinear form on $H^1(X_0,\, \underline{\text{ad}}(F_G))$
given by the latter pairing and the isomorphism in \eqref{e14} coincides with the bilinear from
on $H^1(X_0,\, \underline{\text{ad}}(F_G))$ constructed in \eqref{e13}.

\subsection{Equality of forms}

The biholomorphism ${\mathcal M}^0_X(G)\, \longrightarrow\, {\mathcal R}^0(G)$
obtained by restricting the map $\Phi$ in \eqref{e4} to ${\mathcal M}^0_X(G)\, \subset\,
{\mathcal M}_X(G)$ will be denoted by $\Phi_0$.

\begin{theorem}\label{thm1}
For the form $\Theta_{\mathcal R}$ in \eqref{e11}, the pullback $\Phi^*_0\Theta_{\mathcal R}$ coincides
with the form $\Theta$ on ${\mathcal M}^0_X(G)$ in Lemma \ref{lem1}.
\end{theorem}

\begin{proof}
Take any $(E_G,\, D)\, \in\, {\mathcal M}^0_X(G)$. Let
\begin{equation}\label{de}
d\Phi_0 (E_G,D)\, :\, T_{(E_G,\, D)} {\mathcal M}^0_X(G)\, \longrightarrow\,
T_{\Phi_0(E_G,D)} {\mathcal R}^0(G)
\end{equation}
be the differential of the map $\Phi_0$ at the point $(E_G,\, D)$. We shall describe 
this homomorphism $d\Phi_0(E_G,D)$.

As before, ${\mathcal D}$ denotes the holomorphic connection on $\text{ad}(E_G)$ induced by the holomorphic 
connection $D$ on $E_G$. The Dolbeault operator on the holomorphic vector bundle $\text{ad}(E_G)$ will be 
denoted by $\overline{\partial}'$. The Dolbeault operator on the holomorphic vector bundle 
$\text{ad}(E_G)\otimes K_X$ will be denoted by $\overline{\partial}'_1$.

Consider the Dolbeault resolution of the complex ${\mathcal C}_{\bullet}$ in \eqref{e7}:
\begin{equation}\label{fr}
\begin{matrix}
0&& 0\\
\Big\downarrow && \Big\downarrow\\
\text{ad}(E_G) & \stackrel{\mathcal D}{\longrightarrow} & \text{ad}(E_G)\otimes K_X\\
\Big\downarrow && \Big\downarrow\\
\Omega^{0,0}_X(\text{ad}(E_G))& \stackrel{\mathcal D}{\longrightarrow} & \Omega^{1,0}_X
(\text{ad}(E_G))\\
~ \Big\downarrow \overline{\partial}'&& ~\Big\downarrow\overline{\partial}'_1\\
\Omega^{0,1}_X (\text{ad}(E_G))
& \stackrel{{\mathcal D}'}{\longrightarrow} & \Omega^{1,1}_X(\text{ad}(E_G))\\
\Big\downarrow && \Big\downarrow\\
0&& 0
\end{matrix}
\end{equation}
where ${\mathcal D}'$ is constructed using $\mathcal D$ and the differential operator $\partial$ on
$(0,1)$-forms on $X$; more precisely, ${\mathcal D}'(s\otimes w)\,=\, {\mathcal D}(s)\wedge w+ s\wedge
\partial (w)$, where $s$ is a locally defined section of $\text{ad}(E_G)$ and $w$ is a locally defined
$(0,\, 1)$-form. Note that $\overline{\partial}'_1\circ \mathcal D +{\mathcal D}'\circ\overline{\partial}'$
is the curvature of the connection ${\mathcal D}$ on $\text{ad}(E_G)$. But ${\mathcal D}$ is a holomorphic
connection because $D$ is so, and hence ${\mathcal D}$ is flat. This implies that
\begin{equation}\label{ivo}
\overline{\partial}'_1\circ \mathcal D +{\mathcal D}'\circ\overline{\partial}'\,=\,0\, .
\end{equation}

We observe that the differential operator ${\mathcal D}'$ extends naturally to
the direct sum $$\Omega^{0,1}_X (\text{ad}(E_G))\oplus
\Omega^{1,0}_X (\text{ad}(E_G))\, ,$$ but it is identically zero on $\Omega^{1,0}_X (\text{ad}(E_G))$, because
the sheaf of $(2,\, 0)$-forms on $X$ is the zero sheaf. Also, the differential operator
$\overline{\partial}'_1$ naturally extends to
$$(\text{ad}(E_G)\otimes K_X)\oplus \Omega^{0,1}_X (\text{ad}(E_G))\, ,$$ but
it is identically zero on $\Omega^{0,1}_X (\text{ad}(E_G))$, because
the sheaf of $(0,\, 2)$-forms on $X$ is the zero sheaf.

Using these together with \eqref{ivo}, from the resolution in \eqref{fr} we have the complex of vector spaces
\begin{equation}\label{e15}
0\, \longrightarrow\, C^\infty(X,\, \text{ad}(E_G))\, \stackrel{\overline{\partial}'\oplus
{\mathcal D}}{\longrightarrow}\,C^\infty(X,\, \Omega^{0,1}_X(\text{ad}(E_G)))\oplus
C^\infty(X,\, \text{ad}(E_G)\otimes K_X)
\end{equation}
$$
\, \stackrel{{\mathcal D}'+{\overline{\partial}'_1}}{\longrightarrow}\,
C^\infty(X,\, \Omega^{1,1}_X(\text{ad}(E_G))) \, \longrightarrow\, 0\, .
$$
Since \eqref{fr} is a fine resolution of ${\mathcal C}_{\bullet}$, for \eqref{e15}, we have
\begin{equation}\label{e16}
{\mathbb H}^1({\mathcal C}_{\bullet})\,=\, \frac{\text{kernel}({\mathcal D}'+
\overline{\partial}'_1)}{(\overline{\partial}'\oplus {\mathcal D})(C^\infty(X,\, \text{ad}(E_G)))}\, .
\end{equation}

To describe the target space $T_{\Phi_0(E_G,D)} {\mathcal R}^0(G)$ of the homomorphism
in \eqref{de}, note
that
$$
\nabla\, :=\, {\mathcal D}+ \overline{\partial}'\, :\, \text{ad}(E_G)\,\longrightarrow\,
\text{ad}(E_G)\otimes T^*_{\mathbb R}X
$$
is a flat connection on the vector bundle $\text{ad}(E_G)$ (its curvature is
$\overline{\partial}'_1\circ \mathcal D +{\mathcal D}'\circ\overline{\partial}'$; see
\eqref{ivo}). In fact, it is the one induced by the
flat connection on the principal $G$--bundle $E_G$ associated to the holomorphic
connection $D$. As in \eqref{d1}, consider the complex of vector spaces
$$
0\, \longrightarrow\, C^\infty(X,\, \text{ad}(E_G))
\, \stackrel{\nabla}{\longrightarrow}\, C^\infty(X,\, \text{ad}(E_G)\otimes T^*_{\mathbb R}X)
$$
$$
\stackrel{\nabla}{\longrightarrow}\, C^\infty(X,\, \text{ad}(E_G)\otimes \bigwedge\nolimits^2
T^*_{\mathbb R}X)\, \longrightarrow\, 0\, .
$$
Now as in \eqref{e14} we have
\begin{equation}\label{e17}
T_{\Phi_0(E_G,D)} {\mathcal R}^0(G)\,=\, \frac{\text{kernel}(\nabla)}{\nabla(C^\infty(X,\,
\text{ad}(E_G)))}\, .
\end{equation}

Let
\begin{equation}\label{xi}
\xi\, :\, \frac{\text{kernel}({\mathcal D}'+{\overline{\partial}}'_1)}{(\overline{\partial}'
\oplus {\mathcal D})(C^\infty(X,\, \text{ad}(E_G)))}\,\longrightarrow\,
\frac{\text{kernel}(\nabla)}{\nabla(C^\infty(X,\, \text{ad}(E_G)))}
\end{equation}
be the homomorphism between the two quotient spaces in \eqref{e16} and \eqref{e17} defined by
$$
(\omega_1,\, \omega_2)\, \longmapsto\, \omega_1+\omega_2\, ,
$$
where $\omega_1\, \in\, C^\infty(X,\, \Omega^{0,1}_X(\text{ad}(E_G)))$ and
$\omega_2\, \in\, C^\infty(X,\, \text{ad}(E_G)\otimes K_X)$; note that $\xi$
defines a homomorphism between the quotient spaces.

Using the isomorphisms in \eqref{e16} and \eqref{e17}, the differential
$d\Phi_0 (E_G,D)$ in \eqref{de} coincides with $\xi$ constructed in \eqref{xi}.

We shall now describe the two--form $\Theta(E_G,D)$ (constructed in \eqref{e10})
on ${\mathbb H}^1({\mathcal C}_{\bullet})$ in terms of the isomorphism in \eqref{e16}.

Consider the bilinear form on $C^\infty(X,\, \Omega^{0,1}_X(\text{ad}(E_G)))\oplus
C^\infty(X,\, \text{ad}(E_G)\otimes K_X)$ defined by
$$
(\omega_1,\, \omega_2)\otimes (\omega'_1,\, \omega'_2)\, \longmapsto\,
\int_X \widetilde{B}(\omega_1\wedge \omega'_2) +
\int_X \widetilde{B}(\omega_2\wedge \omega'_1)\, ,
$$
where $\omega_1,\, \omega'_1\, \in\, C^\infty(X,\, \Omega^{0,1}_X(\text{ad}(E_G)))$
and $\omega_2,\, \omega'_2\, \in\, C^\infty(X,\, \text{ad}(E_G)\otimes K_X)$. This
pairing descends to a pairing on the quotient space 
$$
\frac{\text{kernel}({\mathcal D}'+
{\overline{\partial}}'_1)}{(\overline{\partial}'\oplus {\mathcal D})(C^\infty(X,\, \text{ad}(E_G)))}
$$
in \eqref{e16}. The resulting bilinear form on ${\mathbb H}^1({\mathcal C}_{\bullet})$
given by this descended form using the isomorphism in \eqref{e16} actually coincides with
the two--form $\Theta(E_G,D)$ constructed in \eqref{e10}.

In view of the above description of $\Theta(E_G,D)$, using the earlier observation that the 
differential $d\Phi_0 (E_G,D)$ in \eqref{de} coincides with $\xi$ constructed in \eqref{xi},
along with the observation at the end of Section \ref{se3.2} that the form in \eqref{o12} 
descends to the one in \eqref{e13}, it follows that the map $\Phi_0$ takes $\Theta_{\mathcal 
R}(\Phi_0(E_G,D))$ to $\Theta(E_G,D)$.
\end{proof}

Since the form $\Theta_{\mathcal R}$ is closed (it is in fact a symplectic form), and $\Theta$ is 
fiber-wise nondegenerate (Lemma \ref{lem1}), Theorem \ref{thm1} has the following corollary.

\begin{corollary}\label{cor1}
The two-form $\Theta$ on ${\mathcal M}^0_X(G)$ is a symplectic form.
\end{corollary}

Since the form $\Theta$ in Lemma \ref{lem1} is algebraic, Theorem \ref{thm1} has the following
corollary.

\begin{corollary}\label{cor2}
The pulled back form $\Phi^*_0\Theta_{\mathcal R}$ on ${\mathcal M}^0_X(G)$ is algebraic.
\end{corollary}

\section*{Acknowledgements}

The author is very grateful to the referee for helpful comments.
The author thanks Gautam Bharali and Subhojoy Gupta heartily for very helpful comments.
Thanks are due to the Indian Institute of Science for hospitality while the work was
carried out. The author is supported by a J. C. Bose Fellowship.

%%%%%%%%%%%%%%%%%%%%%%%%%%%%%%%%%%%%%%%%%%%%%%%%%%%%%%%%%%%%%%%%%%%%% 

\end{document}